\documentclass{amsart}
\usepackage{amssymb,amstext,amsmath,amscd,amsthm,amsfonts,enumerate,graphicx,latexsym,stmaryrd,multicol}
\usepackage[usenames]{color}
\usepackage{setspace}
\usepackage{bm}
\usepackage[all]{xy}
\usepackage{soul}

\newtheorem{thm}{Theorem}[section]
\newtheorem{lem}[thm]{Lemma}
\newtheorem{prop}[thm]{Proposition}
\newtheorem{cor}[thm]{Corollary}
\theoremstyle{definition}

\newtheorem{rem}[thm]{Remark}

\newtheorem{conv}[thm]{Convention}

\theoremstyle{remark}

\newtheorem*{ac}{Acknowledgments}
\newtheorem*{proof of claim}{Proof of Claim}

\numberwithin{equation}{thm}
\def\A{\mathcal{A}}

\def\assh{{\rm Assh}}
\def\Ann{\mathsf{Ann}}
\def\ann{\mathsf{ann}}
\def\C{\mathsf{C}}
\def\lc{\operatorname{\underline{\mathsf{C}}}}
\def\ca{\mathsf{ca}}
\def\cl{\mathsf{cl}}
\def\cm{\mathsf{CM}}
\def\lcm{\operatorname{\underline{\mathsf{CM}}}}
\def\cok{\operatorname{Coker}}
\def\Deep{\mathsf{Deep}}
\def\ldeep{\operatorname{\underline{\mathsf{Dee}}\mathsf{p}}}
\def\depth{\operatorname{depth}}
\def\dim{\operatorname{dim}}
\def\Ext{\operatorname{Ext}}

\def\height{\operatorname{ht}}
\def\Hom{\operatorname{Hom}}
\def\lhom{\operatorname{\underline{Hom}}}
\def\image{\operatorname{Im}}
\def\ker{\operatorname{Ker}}
\def\m{\mathfrak{m}}
\def\mod{\operatorname{mod}}
\def\lmod{\operatorname{\underline{mod}}}

\def\p{\mathfrak{p}}
\def\q{\mathfrak{q}}
\def\spec{\operatorname{Spec}}
\def\sing{\operatorname{Sing}}
\def\V{\mathrm{V}}
\def\X{\mathcal{X}}
\def\Y{\mathcal{Y}}

\begin{document}
\allowdisplaybreaks
\title{Compactness of the Alexandrov topology of maximal Cohen--Macaulay modules}
\author{Kaito Kimura}
\address{Graduate School of Mathematics, Nagoya University, Furocho, Chikusaku, Nagoya 464-8602, Japan}
\email{m21018b@math.nagoya-u.ac.jp}
\thanks{The author was partly supported by Grant-in-Aid for JSPS Fellows Grant Number 23KJ1117.}

\begin{abstract}
Let $R$ be a Cohen--Macaulay local ring.
In this paper, we first describe the radicals of annihilators of stable categories of maximal Cohen--Macaulay $R$-modules.
We then prove that the Alexandrov topology of the stable category of maximal Cohen--Macaulay $R$-modules is compact provided that the completion of $R$ has an isolated singularity.
Finally, we consider the case of a hypersurface of countable CM-representation type.
\end{abstract}
\maketitle
\section{Introduction}

Throughout the present paper, all rings are assumed to be commutative and noetherian.
The main purpose of this paper is to investigate the annihilators of several stable categories and explore the relationship between their Alexandrov topologies, the cohomology annihilator, and the singular locus.
Our first main result is the theorem below.

\begin{thm}\label{main result general}
Let $(R,\m)$ be a Cohen--Macaulay local ring. Then
$$
\sqrt{\smash[b]{\Ann(\lcm(R))}}
=\sqrt{\smash[b]{\Ann(\lcm_0(R))}}
=\sqrt{\smash[b]{\ca(R)}}
=\bigcap_{\p\in\sing(\widehat{R})} (\p\cap R).
$$
\end{thm}

\noindent
Here, we denote by $\cm(R)$ the category of maximal Cohen--Macaulay $R$-modules, and by $\cm_0(R)$ the category of maximal Cohen--Macaulay $R$-modules that are locally free on the punctured spectrum. 
For $\X\in\{\cm(R), \cm_0(R)\}$ we denote by $\Ann(\underline\X)$ the {\em annihilator} of $\underline\X$, which is by definition the ideal of $R$ consisting of elements $a$ such that the multiplication map by $a$ of each module in $\X$ factors through a free module. 
Also, $\ca(R)$ stands for the {\em cohomology annihilator} of $R$ in the sense of Iyengar and Takahashi \cite{IT}, which is defined as the ideal consisting of elements $a$ such that there exists an integer $n$ with $a\Ext_R^n(M,N)=0$ for all finitely generated modules $M,N$. 
Even when $R$ is not Cohen--Macaulay, considering suitable modules corresponding to maximal Cohen--Macaulay modules makes the first and third equalities hold; see Proposition \ref{main lemma of general}.
On the other hand, the Cohen--Macaulayness is required for the second equality.
In fact, under a few assumptions, the converse of this proposition holds; see Corollary \ref{CMiff}.
Our second main result is an application of the above theorem.

\begin{thm}\label{main cor of TFAE}
Let $(R,\m)$ be a non-regular Cohen--Macaulay local ring. The following are equivalent: 
\begin{enumerate}[\rm(1)]
\item The Alexandrov space $\A(\lcm_0(R))$ is compact;
\item The dimension of $\cm_0(R)$ is finite;
\item The cohomology annihilator $\ca(R)$ is $\m$-primary;
\item The completion $\widehat{R}$ has an isolated singularity.
\end{enumerate}
Furthermore, $\A(\lcm(R))$ is compact if one of these conditions holds.
\end{thm}

We explain our results including the above displayed theorems, relating them with various results in the literature.
Let $R$ be a local ring.
Iyengar and Takahashi \cite{IT} proved that $\ca(R)$ is a defining ideal of $\sing(R)$ if the category $\mod(R)$ of finitely generated $R$-modules has a strong generator, and that $\mod(R)$ has a strong generator if $R$ is an equicharacteristic excellent ring or a localization of a finitely generated algebra over a field.
Recently, Dey, Lank, and Takahashi \cite{DLT} showed that if $R$ is quasi-excellent, then $\mod(R)$ has a strong generator.
We show in Proposition \ref{main lemma of general} that $\ca(R)$ is a defining ideal of $\sing(R)$ if all the formal fibers of $R$ are regular, and confirm in Remark \ref{non reg non trivial}(1) that it does not hold in general.
Dao and Takahashi \cite{DaoT} proved that if $R$ is Cohen--Macaulay and the dimension of $\cm_0(R)$ is finite, then $R$ has an isolated singularity, and that the converse holds true if $R$ is complete, equicharacteristic, and with perfect residue field.
Dey and Takahashi \cite{DeyT} showed the converse when $R$ is excellent, equicharacteristic, and admits a canonical module.
Theorem \ref{main cor of TFAE} partly refines these results since $R$ has an isolated singularity if and only if so does $\widehat{R}$ when all the formal fibers of $R$ are regular.
The equivalence of (2) and (3) in Theorem \ref{main cor of TFAE} is a direct corollary of Theorem \ref{main result general}, which is already known in \cite{DeyT}.

Esentepe \cite{E} showed that the annihilator of $\lcm(R)$ is equal to $\ca(R)$ if $R$ is Gorenstein.
For any Cohen--Macaulay complete equicharacteristic local ring $R$ with perfect residue field, it is shown in \cite{DaoT, DeyT} that the annihilators of $\lcm(R)$ and $\lcm_0(R)$ are defining ideals of  $\sing(R)$.
Theorem \ref{main result general} asserts that the radicals of $\ca(R)$ and the annihilators of $\lcm(R)$ and $\lcm_0(R)$ are always equal.
As mentioned above, they are defining ideals of $\sing(R)$ if all the formal fibers of $R$ are regular.
For a Gorenstein local ring $R$, Akdenizli, Aytekin, \c{C}etin, and Esentepe studied the Alexandrov space $\A(\lcm(R))$ of $\lcm(R)$ and provided a necessary and sufficient condition for $\A(\lcm(R))$ to be compact, in terms of the annihilator of $\lcm(R)$.
Applying the results of Esentepe \cite{E}, they showed that $\A(\lcm(R))$ is compact if $R$ is a complete reduced Gorenstein local ring of dimension one.
Theorem \ref{main cor of TFAE} highly improves this result.

Some of the above results essentially require the (quasi-)excellence of the ring.
It is worth mentioning that it is possible to compare certain annihilators of categories without imposing (quasi-)excellence by completion, as stated in Theorem \ref{main result general}.
The investigation of the compactness of the Alexandrov space in connection with other notions appearing in Theorem \ref{main cor of TFAE} is an original idea of this paper.

It is easy to see that the Alexandrov space is compact in the case of finite CM-representation type.
On the other hand, the compactness in the case of countable CM-representation type is not generally known.
Our third main result addresses this case.

\begin{thm}\label{main example countable}
Let $R$ be a complete local hypersurface with uncountable algebraically closed coefficient field of characteristic different from 2.
If $R$ has countable CM-representation type, $\A(\lcm(R))$ is compact.
\end{thm}
\noindent
This theorem is proved by calculating the annihilators of the endomorphism rings of all objects in $\lcm(R)$.
The ring $R$ appearing in Theorem \ref{main example countable} does not have an isolated singularity.
According to Theorem \ref{main cor of TFAE}, $\A(\lcm_0(R))$ is not compact.
This fact can be directly confirmed by calculation.

The organization of this paper is as follows. 
In section 2, we give the relationship between annihilators of some stable categories and prove Theorems \ref{main result general} and \ref{main cor of TFAE}.
In section 3, we study the Alexandrov topology over a complete local hypersurface with uncountable algebraically closed coefficient field of characteristic different from 2 and show Theorem \ref{main example countable}.

\section{Annihilators of stable categories}

In this section, we consider the annihilator ideals of several stable categories. 
We see that those ideals characterize the singular locus and the Cohen--Macaulayness of a local ring.
First of all, we state the definitions of notions used in this paper.

\begin{conv}\label{convention annihilator}
Throughout this paper, we assume the following.
Let $R$ be a local ring.

{\rm (1)} Denote by $\mod(R)$ the category of finitely generated $R$-modules.
For an $R$-module $M$, the annihilator ideal of $M$ is denoted by $\ann_R(M)$.
Following Iyengar and Takahashi \cite{IT}, we consider the following ideals for any integer $n\ge 0$:
$$
\ca^n(R):=\bigcap_{m\ge n, M,N\in\mod(R)} \ann_R(\Ext_R^m(M,N)) \quad {\rm and} \quad \ca(R):=\bigcup_{m\ge 0}\ca^m(R).
$$
The ideal $\ca(R)$ is called the \textit{cohomology annihilator} of $R$.
Since $R$ is noetherian, the ascending chain $\ca^0(R)\subseteq\ca^1(R)\subseteq\ca^2(R)\subseteq\cdots$ of ideals of $R$ is eventually stable.
Hence $\ca(R)=\ca^n(R)$ for some $n$.

{\rm (2)} The \textit{singular locus} $\sing(R)$ of $R$ is defined as the set of prime ideals of $R$ such that $R_\p$ is not a regular local ring.
We say that $R$ has an \textit{isolated singularity} if any non-maximal prime ideal of $R$ is not in $\sing(R)$.
For each ideal $I$ of $R$, the set of prime ideals of $R$ containing $I$ is denoted by $\V(I)$.
Let $M$ be a finitely generated $R$-module and $F=(\cdots\to F_1\xrightarrow{\alpha}F_0\to0)$ a minimal free resolution of $M$.
The {\em syzygy} $\Omega_R M$ of $M$ is defined as $\operatorname{Im} \alpha$.
We put $\Omega_R^0M=M$, and the {\em $n$th syzygy} $\Omega_R^n M$ is defined inductively by $\Omega_R^n M=\Omega_R(\Omega_R^{n-1}M)$.
The completions of $R$ and $M$ are denoted by $\widehat{R}$ and $\widehat{M}$, respectively.

{\rm (3)} An $R$-module $M$ is called \textit{maximal Cohen--Macaulay} if $\depth M_\p\ge \dim R_\p$ for all prime ideals $\p$ of $R$.
The category of maximal Cohen--Macaulay $R$-modules is denoted by $\cm(R)$.
Respectively, following Takahashi \cite{Ta} and Dao, Eghbali, and Lyle \cite{DEL}, we denote by $\C(R)$ the subcategory of $\mod(R)$ consisting of modules $M$ such that $\depth M_\p\ge \depth R_\p$ for all prime ideals $\p$ of $R$ and by $\Deep(R)$ the subcategory of $\mod(R)$ consisting of modules $M$ such that $\depth M \ge \depth R$.
An $R$-module $M$ is \textit{locally free on the punctured spectrum} if $M_\p$ is a free $R_\p$-module for any non-maximal prime ideal $\p$ of $R$.
For a subcategory $\X$ of $\mod(R)$, denote by $\X_0$ the subcategory of $\X$ consisting of modules which are locally free on the punctured spectrum. 
Note that $\C_0(R)=\Deep_0(R)$ and that if $R$ is Cohen--Macaulay, then $\cm(R)=\C(R)=\Deep(R)$.

{\rm (4)}
For $R$-modules $M,N$, denote by $\mathcal{P}_R(M,N)$ the set of $R$-homomorphisms from $M$ to $N$ factoring through projective modules, which an $R$-submodule of $\Hom_R(M,N)$.
Let us denote by $\lhom_R(M,N)$ the quotient $R$-module $\Hom_R(M,N)/\mathcal{P}_R(M,N)$.
For a subcategory $\X$ of $\mod(R)$, the \textit{stable category} $\underline\X$ of $\X$ is defined as the category whose objects are the same as $\X$ and whose morphism sets are $\lhom_R(M,N)$.

{\rm (5)} Let $S$ be a set.
A \textit{preorder} or a \textit{quasiorder} on $S$ is a binary relation $\le$ on $S$ that is reflexive and transitive.
The \textit{closure} $\cl(x)$ of a point $x\in S$ is a subset $\{y\in S\mid y\le x\}$ of $S$, and the closure $\cl(U)$ of a subset $U$ of $S$ is the union $\bigcup_{x\in U}\cl(x)$.
There is a topology on $S$ for which $\{\cl(U)\mid U \subseteq S\}$ is the set of closed sets.
This topology is known as \textit{Alexandrov topology} associated to $\le$.
Let $\X$ be an $R$-linear category (having a small skeleton). 
For any object $X\in\X$, we consider the following ideals:
$$
\Ann(X):=\ann_R \Hom_\X(X,X) \quad {\rm and} \quad \Ann(\X):=\bigcap_{X\in\X} \Ann(X).
$$
The inclusion relation $\Ann(X)\subseteq\Ann(Y)$ for $X,Y\in\X$ defines a preorder on the set of isomorphism classes of objects of $\X$.
The corresponding Alexandrov space is denoted by $\A(\X)$; see \cite{AACE}.
Note the distinction between uppercase $\Ann$ and lowercase $\ann$.
In this paper, quasi-compact is simply called compact, that is, we say that a topological space $T$ is compact if every open cover has a finite subcover.

{\rm (6)} All subcategories are nonempty, full, and closed under isomorphism.
Let $\X, \Y$ be subcategories of $\mod(R)$.
Denote by $[\X]$ the smallest subcategory of $\mod(R)$ containing $\{R\}\cup\X$ that is closed under direct summands, finite direct sums, and syzygies.
We denote by $\X\circ\Y$ the subcategory of $\mod(R)$ consisting of objects $Z$ such that there exists an exact sequence $0\to X\to Z\to Y\to 0$ in $\mod(R)$ with $X\in\X$ and $Y\in\Y$.
We set $[\X]_1=[\X]$, and the \textit{ball of radius} $n$ centered at $\X$ is defined inductively by $[\X]_n=[[\X]_{n-1}\circ[\X]]$.
Following Dao and Takahashi \cite{DaoT}, we define the \textit{dimension} of $\X$, denoted by $\dim(\X)$, as the infimum of $n\ge 0$ such that $\X=[G]_{n+1}$ for some object $G\in\X$.
\end{conv}

We prepare two lemmas that play crucial roles in the proof of the main results of this section.

\begin{lem}\cite[Lemma 3.8]{DeyT}\label{DeyT3.8}
Let $R$ be a local ring, and $M$ an $R$-module. Then
$$
\bigcap_{m>0, N\in\mod(R)} \ann_R(\Ext_R^m(M,N))=\ann_R(\Ext_R^1(M,\Omega M))=\ann_R(\lhom_R(M,M)).
$$
\end{lem}

Let $R$ be a local ring.
For any $M,N \in\mod (R)$ and $n\ge 0$, one has $\Ext_R^{n+1}(M,N)\cong\Ext_R^1(\Omega^n M,N)$.
Lemma \ref{DeyT3.8} implies that $\ca^{n+1}(R)$ is equal to $\Ann(\underline{\Omega^n}(R))$, where $\underline{\Omega^n}(R)$ is the subcategory of $\lmod(R)$ consisting of $n$th syzygy modules.
In particular, $\ca(R)=\Ann(\underline{\Omega^n}(R))$ for some $n\ge 0$.

\begin{lem}\cite[Corollary 4.4(2)]{DaoT}\label{DaoT4.4(2)}
Let $R$ be a local ring of dimension $d$, $a$ an element of $R$, and $n$ an integer.
Suppose that $a\Ext_R^i(M,N)=0$ for all finitely generated $R$ modules $M,N$ having finite length and all $n\le i\le n+2d$.
Then $a^{2^{2d}}\Ext_R^i(M,N)=0$ for all finitely generated $R$ modules $M,N$.
\end{lem}

The proposition below says that some equalities exist for the radicals of annihilators of stable categories.
There are many studies on the relationship between $\ca(R)$ and the definition ideal of $\sing (R)$; see \cite{BHST, IT, L} for instance.
Also, the annihilators of $\lcm(R)$ and $\lcm_0 (R)$ are defining ideals of $\sing (R)$ under several assumptions; see \cite{DaoT, DeyT}.
Proposition \ref{main lemma of general} describes the behavior of those ideals in general and improves some of their results; see Remark \ref{non reg non trivial}(1).

\begin{prop}\label{main lemma of general}
Let $(R, \m)$ be a local ring of dimension $d$.
Put $I=\Ann(\lc_0(R))=\Ann(\ldeep_0(R))$.
\begin{enumerate}[\rm(1)]
\item There is an equality $\sqrt{\smash[b]{\ca(R)}}=\bigcap_{\p\in\sing(\widehat{R})} (\p\cap R)$.
\item If $R$ has a positive depth, then there are equalities $I=\Ann(\lc(R))=\Ann(\ldeep(R))$, otherwise one has $I^{d+1}\subseteq\Ann(\ldeep(R))\subseteq\Ann(\lc(R))\subseteq I$.
In particular, $\sqrt{\smash[b]{I}}=\sqrt{\smash[b]{\Ann(\ldeep(R))}}=\sqrt{\smash[b]{\Ann(\lc(R))}}$.
\item Suppose that $R$ is a Cohen--Macaulay local ring.
Then $\sqrt{\smash[b]{\Ann(\lcm(R))}}=\sqrt{\smash[b]{\ca(R)}}$.
\end{enumerate}
\end{prop}

\begin{proof}
(1) Since $\widehat{R}$ is a noetherian quasi-excellent ring of finite dimension, \cite[Corollary C]{DLT} implies that $\mod(\widehat{R})$ admits a strong generator.
By \cite[Theorem 1.1]{IT}, we have $\V(\ca(\widehat{R}))=\sing(\widehat{R})$, and hence
$$
\sqrt{\smash[b]{\ca(\widehat{R})\cap R}}=\bigcap_{\p\in\sing(\widehat{R})} (\p\cap R).
$$
We claim $\sqrt{\smash[b]{\ca(R)}}=\sqrt{\smash[b]{\ca(\widehat{R})\cap R}}$.
There exists $n$ such that $\ca^n(R)=\ca(R)$ and $\ca^n(\widehat{R})=\ca(\widehat{R})$.
It follows from \cite[Theorem 4.5]{BHST} that $\ca^n(\widehat{R})\cap R\subseteq\ca^n(R)$.
Let $M,N$ be finitely generated $\widehat{R}$-modules having finite length, and let $i\ge n$. 
These modules also have finite length as $R$-modules, and there are isomorphisms $\Ext_{R}^i(M,N)\cong\Ext_{R}^i(M,N)\otimes_R \widehat{R}\cong\Ext_{\widehat{R}}^i(M,N)$ of $R$-modules.
Any element of $R$ belonging to $\ca^n(R)$ annihilates $\Ext_{\widehat{R}}^i(M,N)$.
By Lemma \ref{DaoT4.4(2)}, we obtain $\ca^n(R)^{2^{2d}}\subseteq\ca(\widehat{R})\cap R$.

(2) Since $\C_0(R)\subseteq\C(R)\subseteq\Deep(R)$, we have $\Ann(\ldeep(R))\subseteq\Ann(\lc(R))\subseteq I$.
We have only to show that if $R$ has a positive depth, $\Ann(\ldeep(R))$ contains $I$, otherwise $\Ann(\ldeep(R))$ contains $I^{d+1}$.
Let $a\in I$ and $M\in\Deep(R)$.
Suppose that $M_\p$ is a free $R_\p$-module for any ideal $\p$ of $R$ such that $\height\p<n$.
We claim by descending induction on $n$ that if $R$ has a positive depth, $a$ annihilates $\Ext_R^1(M,\Omega M)$, otherwise $a^{d-n+1}$ annihilates $\Ext_R^1(M,\Omega M)$.
In the case $n=d$, $M$ belongs to $\Deep_0(R)$, which means that the claim holds.
Let $0\le n<d$.
The nonfree locus $V:=\{\p\in\spec(R) \mid M_\p$ is not free as an $R_\p$-module$\}$ of $M$ is a closed subset of $\spec(R)$; see \cite[Theorem 4.10]{Mat}.
There are at most a finite number of prime ideals of height $n$ belonging to $V$ by assumption.
We can choose $x\in\m$ such that it does not belong to any prime ideal of height $n$ belonging to $V$ and any non-maximal associated prime ideal of $M$.
Let $i>0$.
If $\depth R=0$, then $\depth \Omega(M/x^i M)\ge 0=\depth R$, otherwise $\depth M\ge\depth R>0$ and thus $x$ is an $M$-regular element, which implies $\depth \Omega(M/x^i M)\ge\depth R$.
Since $x$ is $M_\p$-regular for any non-maximal prime ideal $\p$, we see that $\Omega(M/x^i M)$ is in $\Deep(R)$ and $\Omega(M/x^i M)_\p$ is a free $R_\p$-module for any ideal $\p$ such that $\height\p<n+1$.

We consider the case where $R$ has a positive depth.
For any positive integer $i$, there is an exact sequence $0\to M\xrightarrow{x^i} M\to M/x^i M\to 0$, which induces an exact sequence 
$$
\Ext_{R}^1(M,\Omega M)\xrightarrow{x^i} \Ext_{R}^1(M,\Omega M)\to \Ext_R^2(M/x^i M,\Omega M)\cong\Ext_{R}^1(\Omega (M/x^i M),\Omega M).
$$
It follows Lemma \ref{DeyT3.8} and the induction hypothesis that $a$ annihilates $\Ext_{R}^1(\Omega (M/x^i M),\Omega M)$, which implies $a \Ext_{R}^1(M,\Omega M)\subseteq x^i \Ext_{R}^1(M,\Omega M)$.
By Krull's intersection theorem, $a \Ext_{R}^1(M,\Omega M)=0$.

We deal with the case $\depth R=0$.
For any $i>0$, the submodule $(0:x^i)_M$ of $M$ has finite length and hence it belongs to $\Deep_0(R)$.
There is an exact sequence $0\to (0:x^i)_M\to M\xrightarrow{x^i} M\to M/x^i M\to 0$, which induces exact sequences
\begin{align*}
\Ext_{R}^1(x^i M,\Omega M)\to \Ext_{R}^1(M,\Omega M)\to \Ext_R^1((0:x^i)_M,\Omega M) \ {\rm and} \\
\Ext_{R}^1(M,\Omega M)\to \Ext_{R}^1(x^i M,\Omega M)\to \Ext_{R}^1(\Omega (M/x^i M),\Omega M).
\end{align*}
By induction hypothesis, we get $a\Ext_R^1((0:x^i)_M,\Omega M)=0=a^{d-n}\Ext_{R}^1(\Omega (M/x^i M),\Omega M)$.
An analogous argument to the proof of \cite[Theorem 4.3(3)]{DaoT} shows that 
$a^{d-n+1} \Ext_{R}^1(M,\Omega M)\subseteq x^i \Ext_{R}^1(M,\Omega M)$ for all $i>0$, which means $a^{d-n+1} \Ext_{R}^1(M,\Omega M)=0$.

(3) Let $M,N$ be finitely generated $R$-modules, $a$ an element of $R$ belonging to $\Ann(\lcm(R))$, and $i\ge d+1$ an integer.
Since $a$ annihilates $\lhom_R(\Omega^d M, \Omega^d M)$, it also annihilates $\Ext_R^{i-d}(\Omega^d M, N)\cong\Ext_R^{i}(M, N)$ by Lemma \ref{DeyT3.8}, which implies $\Ann(\lcm(R))\subseteq\ca^{d+1}(R)\subseteq\ca(R)$.

Now, we prove the converse inclusion.
Take $n>0$ such that $\ca^{n}(R)=\ca(R)$.
First, we deal with the case where $R$ admits a canonical module $\omega$.
Let $\operatorname{tr}\omega$ be the trace ideal of $\omega$ defined by the image of the canonical map $\Hom_R(M,R)\otimes_R M\to R$ given by $f\otimes x \mapsto f(x)$.
By \cite[Theorem 2.3]{DKT}, we have
$$
\operatorname{tr}\omega=\ann_{R}(\Ext_{R}^1(\omega,\Omega_R \omega)).
$$
On the other hand, for any prime ideal $\p$ that does not contain $\ca^n(R)$, $R_\p$ is regular.
Indeed, $\p$ does not even contain $\ann_{R}(\Ext_{R}^{n}(R/\p, R/\p))$, which means that the residue field of $R_\p$ has projective dimension at most $n-1$.
Then $\omega_\p$ is free, which implies $\operatorname{tr}\omega=\ann_{R}(\Ext_{R}^1(\omega,\Omega_R \omega))\nsubseteq\p$.
We have $\sqrt{\smash[b]{\ca^n(R)}}\subseteq\sqrt{\smash[b]{\operatorname{tr}\omega}}$ and hence $\ca^n(R)^l \subseteq \operatorname{tr}\omega$ for some $l>0$.
It follows from \cite[Lemma 4.1(2)]{DeyT} that for any $M\in\cm(R)$,
$$
(\operatorname{tr}\omega)^{n-1}
\cdot\ann_{R}(\Ext_{R}^n(M,\Omega^n M))
\subseteq \ann_{R}(\Ext_{R}^1(M,\Omega M)).
$$
This means that $\ann_{R}(\Ext_{R}^1(M,\Omega M))$ contains $\ca^n(R)^{l(n-1)+1}$, and so does $\Ann(\lcm(R))$.

Next, we handle the general case.
A similar argument to the proof of (1) shows that 
$$
\sqrt{\smash[b]{\ca(R)}}\subseteq\sqrt{\smash[b]{\ca(\widehat{R})\cap R}}\subseteq\sqrt{\smash[b]{\Ann(\lcm(\widehat{R}))\cap R}}\subseteq\sqrt{\smash[b]{\Ann(\lcm(R))}}
$$
since $\widehat{R}$ admits a canonical module.
\end{proof}

When R is excellent, the proof of (1) is essentially the same as the argument of the proof of \cite[Theorem 5.3]{IT}.
In the proof of (2), the method of Dey and Takahashi \cite{DeyT} plays an essential role.
On the other hand, it is worth mentioning that the comparison of annihilators without imposing specific assumptions is achieved.

\begin{proof}[Proof of Theorem \ref{main result general}]
The assertion is an immediate consequence of Proposition \ref{main lemma of general}.
\end{proof}

Here are two remarks on Proposition \ref{main lemma of general}.

\begin{rem}\label{non reg non trivial}
(1) Let $R$ be a local ring.
We put $V=\{\p\cap R \mid \p\in\sing(\widehat{R})\}$.
It is well known that $\sing(R)$ is closed and equal to $V$ if $R$ is G-ring; see \cite[(33.A) and Theorem 76]{Matsu} for instance.
Proposition \ref{main lemma of general} asserts that if $R$ is G-ring, then $\ca(R)$ is a defining ideal of $\sing(R)$ and so are $\Ann(\lcm(R))$ and $\Ann(\lcm_0(R))$ when $R$ is Cohen--Macaulay.
It refines \cite[Proposition 4.8]{DaoT}, \cite[Theorem 3.12]{DeyT}, and \cite[Theorems 5.3 and 5.4]{IT} in the case of local rings.
Here we describe our observations for $V$.

The minimal elements of $V$ are at most finite since $\sing(\widehat{R})$ is a closed subset of $\spec(\widehat{R})$.
So, $V$ is closed if and only if $V$ is \textit{stable under specialization}, that is, if $\q\in V$ and $\q'\in \spec(R)$ with $\q\subseteq \q'$, then $\q'\in V$.
Proposition \ref{main lemma of general}(1) asserts that $V$ is contained in $\V(\ca(R))$ and any prime ideal being in $\V(\ca(R))$ contains some minimal elements of $V$, which implies that $\V(\ca(R))$ is equal to the closure of $V$.
Hence $V$ is closed if and only if $\ca(R)$ is a defining ideal of $V$.
Note that $V$ contains $\sing(R)$; see \cite[Theorem 2.2.12(a)]{BH}.
It is easy to see that $\sing(R)$ is equal to $V$ if and only if the following condition is satisfied: for every $\p\in\spec(\widehat{R})$ and $\q=\p\cap R$, $\p$ belongs to $\sing(\widehat{R})$ if and only if $\q$ belongs to $\sing(R)$.
By \cite[Theorem 2.2.12(b)]{BH}, $\sing(R)=V$ if $\kappa(\p)\otimes_R \widehat{R}$ is regular for every prime ideal $\p$ of $R$, where $\kappa(\p)$ is the residue field of $R_\p$.
(This is weaker than the assumption that $R$ is G-ring.)
Also, $\sing(R)=V$ if and only if $\ca(R)$ is a defining ideal of $\sing(R)$.
Indeed, as $\sing(R)$ is stable under specialization, if $V$ is equal to $\sing(R)$, then $V$ is closed and thus we have $V=\V(\ca(R))$.
On the other hand, if $\sing(R)$ is not equal to $V$, then it is not equal to $\V(\ca(R))$ either.
In particular, $\V(\ca(R))\ne\sing(R)$ if $R$ has an isolated singularity and $\widehat{R}$ does not.

(2) The assumption in Proposition \ref{main lemma of general}(3) that the ring is Cohen--Macaulay is essential.
Let $(R, \m)$ be a non-Cohen--Macaulay local ring of depth zero such that $\widehat{R}$ has an isolated singularity.
(For example, a quotient of a formal power series ring $K \llbracket x, y \rrbracket/(x^2, xy)$ over a field $K$ is one such ring.)
Proposition \ref{main lemma of general}(1) implies that $\ca(R)$ is $\m$-primary.
Note that $\C(R)=\mod(R)$ by definition.
Let $n$ be a positive integer, and $a\in\ann_R\lhom(R/\m^n, R/\m^n)$.
Denote by $\tilde{a}$ the multiplication by $a$ on $R/\m^n$.
Then there are $R$-homomorphisms $f:R/\m^n\to R$ and $g:R\to R/\m^n$ such that $\tilde{a}=g\circ f$.
Since $\image(f)$ is contained in $(0:\m^n)_R\subseteq\Gamma_{\m}(R)$, $a$ belongs to $(0:\m^n)_R+\m^n$.
Let $\Gamma_{\m}(R)$ be an $\m$-torsion submodule of $R$.
We have
$$
\Ann(\lc(R))\subseteq\bigcap_{n>0}\ann_R\lhom(R/\m^n, R/\m^n)\subseteq\bigcap_{n>0} ((0:\m^n)_R+\m^n)\subseteq\bigcap_{n>0} (\Gamma_{\m}(R)+\m^n)=\Gamma_{\m}(R).
$$
As $R$ is not artinian, $\Gamma_{\m}(R)$ is not an $\m$-primary ideal.
Therefore, the ideal $\Ann(\lc(R))$ is not $\m$-primary, in other words, $\sqrt{\smash[b]{\Ann(\lc(R))}}\ne \sqrt{\smash[b]{\ca(R)}}$.
\end{rem}

Let us study $\Ann(\lc(R))$ in more detail when $R$ is not Cohen--Macaulay.
For a local ring $R$, the set of prime ideals $\p$ of $R$ such that $\dim (R/\p)=\dim R$ is denoted by $\assh(R)$.
The following proposition is not meaningful for Cohen--Macaulay rings.
Indeed, if $R$ is Cohen--Macaulay, for any prime ideal $\p$ of $R$, the equalities $\dim(R/\p)+\depth R_\p=\dim R -\height\p+\depth R_\p=\depth R$ hold.
On the other hand, one has $\dim(R/\p)+\depth R_\p=\dim R>\depth R$ for any $\p\in\assh(R)$ when $R$ is not Cohen--Macaulay.

\begin{prop}\label{nonCM}
Let $(R,\m)$ be a local ring.
Then $\Ann(\lc(R))\subseteq\p$ for all prime ideals $\p$ of $R$ such that $\dim(R/\p)+\depth R_\p>\depth R$.
\end{prop}

\begin{proof}
We handle the case where $R$ is complete.
We may assume $t:=\depth R<\dim R=:d$.
There is a Gorenstein local ring $S$ of dimension $d$ such that $R$ is a homomorphic image of $S$.
Let $a\in\Ann(\lc(R))$.
For any $n>0$ and $t<i\le d$, $\Omega^t(R/\m^n)$ is in $\lc(R)$ and $a$ annihilates $\Ext^i_R(R/\m^n, R)\cong\Ext^{i-t}_R(\Omega^t(R/\m^n), R)$, which implies that $a$ belongs to $\ann_R H_\m^i(R)=\ann_R \Ext^{d-i}_S(R,S)$; see \cite[Theorem 8.1.1(a)]{BH}.
Let $\p$ be a prime ideal of $R$ such that $\dim(R/\p)+\depth R_\p>\depth R$ and $\q$ the inverse image of $\p$ in $S$.
Note that $\dim(R/\p)=\dim(S/\q)=d-\height\q$ as $S$ is Cohen--Macaulay.
By assumption, we have 
$$
0\le m:=\depth S_\q-\depth R_\p=\height\q-\depth R_\p=d-\dim(R/\p)-\depth R_\p<d-t.
$$
It follows from \cite[Exercises 3.1.24]{BH} that $\Ext^{m}_{S_\q}(R_\p,S_\q)\ne 0$, which means that $\q$ contains $\ann_S \Ext^{m}_S(R,S)$.
Therefore, we get $a\in\ann_R \Ext^{m}_S(R,S)\subseteq\p$.

We address the general case.
For any prime ideal $\p$ of $R$ such that $\dim(R/\p)+\depth R_\p>\depth R$, we can choose a prime ideal $\q$ of $\widehat{R}$ such that $\dim (\widehat{R}/\q)=\dim (\widehat{R}/\p \widehat{R})$.
One has the equality $\p=\q\cap R$.
Indeed, if $\p\subsetneq\q\cap R$, then there is $\q'\in\spec(\widehat{R})$ such that $\q'\subsetneq\q$ and $\p=\q'\cap R$ by \cite[Theorem 9.5]{Mat}.
This means that $\q$ is not a minimal prime ideal of $\p\widehat{R}$, which contradicts the choice of $\q$.
As $\widehat{R}_\q/\p \widehat{R}_\q$ is artinian, we see that $\depth \widehat{R}_\q=\depth R_\p$ by \cite[Proposition 1.2.16(a)]{BH}.
We have
$$
\dim (\widehat{R}/\q)+\depth \widehat{R}_\q=\dim(R/\p)+\depth R_\p>\depth R=\depth \widehat{R}.
$$
It follows from Proposition \ref{main lemma of general}(2) and the first half of this proof that $\Ann(\lc_0(\widehat{R}))$ is contained in $\q$.
Let $a\in\Ann(\lc(R))$.
For any finitely generated $\widehat{R}$-modules $M,N$ having finite length and $i>0$, $a$ annihilates $\Ext_{R}^i(\Omega^t_R M,N)\cong\Ext_{R}^{t+i}(M,N)\cong\Ext_{\widehat{R}}^{t+i}(M,N)$; see the proof of Proposition \ref{main lemma of general}(1).
By \cite[Theorem 2.2]{T}, for any $X\in\C_0(\widehat{R})$, there exists a finitely generated $\widehat{R}$-module $M$ such that $X$ is a direct summand of $\Omega^t M$.
Lemma \ref{DaoT4.4(2)} yields that $a^{2^{2d}}$ annihilates $\Ext_{\widehat{R}}^{t+1}(M,\Omega X)\cong\Ext_{\widehat{R}}^1(\Omega^t M,\Omega X)$, which implies that $\Ann(\lc_0(\widehat{R}))$ contains $a^{2^{2d}}\widehat{R}$.
We have $\Ann(\lc(R))^{2^{2d}}\subseteq\Ann(\lc_0(\widehat{R}))\cap R\subseteq\q\cap R=\p$.
\end{proof}

The corollary below asserts that under a few assumptions, the converse of Theorem \ref{main result general} holds.

\begin{cor}\label{CMiff}
Let $R$ be a local ring.
Suppose that $\widehat{R}_\p$ is regular for some $\p\in\assh(\widehat{R})$.
(For example, $\widehat{R}$ is reduced.)
Then $R$ is Cohen--Macaulay if and only if $\sqrt{\smash[b]{\Ann(\lc(\widehat{R}))}}=\sqrt{\smash[b]{\ca(\widehat{R})}}$.
\end{cor}

\begin{proof}
The ``only if'' part follows from Theorem \ref{main result general}.
In order to prove the ``if'' part, suppose that $R$ is not Cohen--Macaulay.
It follows from Proposition \ref{main lemma of general} (or Remark \ref{non reg non trivial}) that $\p$ does not contain $\ca(\widehat{R})$.
On the other hand, thanks to Proposition \ref{nonCM}, $\p$ contains $\Ann(\lc(\widehat{R}))$.
The proof is now completed.
\end{proof}

In particular, for complete local domains, the above corollary says that the equality in Theorem \ref{main result general} characterizes the Cohen--Macaulayness of the ring.
The inequality $\sqrt{\smash[b]{\Ann(\lc(R))}}\ne \sqrt{\smash[b]{\ca(R)}}$ for the ring $R$ appearing in Remark \ref{non reg non trivial}(2) also follows from Corollary \ref{CMiff}.

\begin{rem}
(1) The equivalence of Corollary \ref{CMiff} does not necessarily hold.
A quotient of a formal power series ring $R=K \llbracket x, y \rrbracket/(x^2y^3, x^3y^2)$ over a field $K$ is a non-Cohen--Macaulay complete local ring satisfying $\sing(R)=\spec(R)$.
The implications $0\subseteq\Ann(\lc(R))\subseteq\ca(R)$ holds; see the proof of Proposition \ref{main lemma of general}(3).
By Proposition \ref{main lemma of general}(1), $\ca(R)$ is nilpotent and thus $\sqrt{\smash[b]{0}}=\sqrt{\smash[b]{\Ann(\lc(R))}}=\sqrt{\smash[b]{\ca(R)}}$.
More generally, a non-Cohen--Macaulay local ring $R$ with $\sing(\widehat{R})=\spec(\widehat{R})$ is a counterexample to the equivalence in Corollary \ref{CMiff}.

(2) Let $(R,\m,k)$ be a local ring of dimension $d$.
If either $\Ann(\lc(R))$, $\Ann(\lc_0(R))$, or $\ca(R)$ is equal to $R$, then so is $\ann_R\Ext^n(\Omega^d(k),\Omega^{n+d}(k))$ for some positive integer $n$; see Lemma \ref{DeyT3.8}.
This means $\Ext^1(\Omega^{n+d-1}(k),\Omega^{n+d}(k))=0$, which implies $\Omega^{n+d-1}(k)$ is free, and thus $R$ is regular.
Conversely, if $R$ is regular, then all modules belonging to $\lc(R)$ are free.
It is easy to see that the equalities $\Ann(\lc(R))=\Ann(\lc_0(R))=\ca^{d+1}(R)=\ca(R)=R$ hold.
\end{rem}

Akdenizli, Aytekin, \c{C}etin, and Esentepe \cite{AACE} studied the Alexandrov topology over a Gorenstein local ring and gave an equivalence between the compactness of that topology and the existence of a minimum element with respect to the preorder.
Theorem \ref{main TFAE} characterize the existence of a minimum element in terms of the annihilators of stable categories, the cohomology annihilator, and the singular locus.

\begin{thm}\label{main TFAE}
Let $(R,\m)$ be a non-regular Cohen--Macaulay local ring. Then the following are equivalent: 
\begin{enumerate}[\rm(1)]
\item $\Ann(\lcm_0(R))=\Ann(M)$ for some $M\in\lcm_0(R)$.
\item $\Ann(\lcm_0(R))$ is $\m$-primary;
\item $\Ann(\lcm(R))$ is $\m$-primary;
\item $\ca(R)$ is $\m$-primary;
\item $\widehat{R}$ has an isolated singularity.
\end{enumerate}
\end{thm}

\begin{proof}
If $\Ann(\lcm_0(R))=\Ann(M)$ for some $M\in\lcm_0(R)$, then $\Ann(\lcm_0(R))$ is $\m$-primary since $M$ is locally free on the punctured spectrum and $\Ann(\lcm_0(R))$ is proper; see Remark \ref{non reg non trivial}.
Hence the implication $(1)\Rightarrow(2)$ holds true.
In order to prove the converse implication, we assume that $\Ann(\lcm_0(R))$ is $\m$-primary.
Take an object $M_0\in\lcm_0(R)$.
For each integer $i\ge 0$, if $\Ann(M_i)\nsubseteq\Ann(N_i)$ for some $N_i\in\lcm_0(R)$, then we put $M_{i+1}=M_i \oplus N_i$.
It is easy to see that
$$
\Ann(M_{i+1})=\Ann(M_i)\cap\Ann(N_i)\subsetneq\Ann(M_i);
$$
see Lemma \ref{DeyT3.8} for instance.
As $R/\Ann(\lcm_0(R))$ is artinian, the descending chain 
$$
\Ann(\lcm_0(R))\subseteq\cdots\subsetneq\Ann(M_{i+1})\subsetneq\Ann(M_i)\subsetneq\cdots\subsetneq\Ann(M_1)\subsetneq\Ann(M_0)
$$
of ideals of $R$ must stop after a finite number steps, which means $\Ann(\lcm_0(R))=\Ann(M_i)$ for some $i$.
The proof of $(2)\Rightarrow(1)$ is complete.
It follows from Theorem \ref{main result general} that $(2)\Leftrightarrow(3)\Leftrightarrow(4)\Leftrightarrow(5)$ hold.
\end{proof}

Theorem \ref{main cor of TFAE} is a direct corollary of the above theorem.

\begin{proof}[Proof of Theorem \ref{main cor of TFAE}]
By \cite[Theorem 2.6]{AACE}, the condition (1) holds if and only if $\Ann(\lcm_0(R))=\Ann(M)$ for some $M\in\lcm_0(R)$.
It follows from Lemma \ref{DeyT3.8} and \cite[Theorem 1.1]{DaoT} that (2) holds true if and only if $\Ann(\lcm_0(R))$ is $\m$-primary.
Hence the equivalence follows from Theorem \ref{main TFAE}.
If $\widehat{R}$ has an isolated singularity, then so does $R$, and thus $\lcm_0(R)=\lcm(R)$.
The last assertion follows from (1).
\end{proof}

Note that $R$ has an isolated singularity if and only if so does $\widehat{R}$ when all the formal fibers of $R$ are regular.
Therefore, we can remove by Theorem \ref{main cor of TFAE} the assumptions imposed in \cite[Theorem 1.2(3)]{DeyT} for the case $c=0$ that the ring is excellent, equicharacteristic, and admits a canonical module.
In particular, Theorem \ref{main cor of TFAE} partly refines \cite[Theorem 1.1(1)]{DaoT}.
Also, Theorem \ref{main cor of TFAE} improves \cite[Example 3.1]{AACE}.

The followin corollary follows immediately from Theorem \ref{main cor of TFAE}.

\begin{cor}\label{G-ring quasi excellent corollary of TFAE}
Let $R$ be a Cohen--Macaulay local ring having an isolated singularity.
If $\kappa(\p)\otimes_R \widehat{R}$ is regular for every prime ideal $\p$ of $R$, where $\kappa(\p)$ is the residue field of $R_\p$, then $\A(\lcm(R))$ is compact.
In particular, if $R$ is {\rm (}quasi-{\rm )}excellent, then $\A(\lcm(R))$ is compact.
\end{cor}

\begin{proof}
If $\kappa(\p)\otimes_R \widehat{R}$ is regular for any prime ideal $\p$ of $R$, then $\widehat{R}$ has an isolated singularity by \cite[Theorem 2.2.12]{BH}.
The assertion is a consequence of Theorem \ref{main cor of TFAE}.
\end{proof}

\section{Countable CM-representation type}

Let us call a local ring $R$ of finite (resp. countable) CM-representation type if there are only finitely (resp. countably) many isomorphism classes of indecomposable maximal Cohen--Macaulay modules over $R$.
It follows from \cite[Theorem 2.6]{AACE} and Lemma \ref{lemma AACE example1} that $\A(\lcm(R))$ is compact when $R$ is of finite CM-representation type.
Also, the condition (2) of Theorem \ref{main cor of TFAE} holds in the case of finite CM-representation type, however, it is unknown whether it holds for countable CM-representation types.
In this section, we investigate the compactness of in the case of countable CM-representation type following the ideas presented in \cite[Section 5]{AACE}.
To calculate examples, we prepare two lemmas.

\begin{lem}\label{lemma AACE example1}
Let $R$ be a local ring and $M$ a maximal Cohen--Macaulay $R$-module.
\begin{enumerate}[\rm(1)]
\item Suppose that $\Ann(M)\subseteq\Ann(N)$ for any indecomposable maximal Cohen--Macaulay $R$-module $N$.
Then $\Ann(M)=\Ann(\lcm(R))$.
\item If $R$ is Gorenstein, then $\Ann(M)=\Ann(\Omega_R M)$.
\end{enumerate}
\end{lem}

\begin{proof}
The proof of \cite[Lemma 2.2]{AACE} shows $\Ann(X\oplus Y)=\Ann(X)\cap\Ann(Y)$ for all maximal Cohen--Macaulay $R$-modules $X,Y$, which implies (1) holds.
The assertion (2) follows since the functor $\Omega_R:\lcm(R)\to \lcm(R)$ is equivalence if $R$ is Gorenstein.
\end{proof}

The following lemma is used in \cite[Section 5]{AACE}.
We denote by $I_n$ the identity matrix of size $n$.

\begin{lem}\label{lemma AACE example2}
Let $R$ be a Gorenstein local ring, and $r$ an element of $R$.
Let $\cdots \xrightarrow{\phi} R^{n} \xrightarrow{\psi} R^{n} \xrightarrow{\phi} R^{n} \xrightarrow{\psi} \cdots$ be an exact sequence of free $R$-modules, and $M=\cok\phi$.
\begin{enumerate}[\rm(1)]
\item The element $r$ belongs to $\Ann(M)$ if and only if there are $R$-homomorphisms $\alpha,\beta:R^{n} \to R^{n}$ such that $\phi\circ\alpha+\beta\circ\psi=r I_n$.
\item Suppose that one of the conditions in {\rm (1)} holds.
If $\phi$ and $\psi$ are given by matrices $(x_{ij})$ and $(y_{ij})$, then $r$ belongs to $(x_{k1}, \ldots, x_{kn}, y_{1k}, \ldots, y_{nk})$ for all $1\le k\le n$.
\end{enumerate}
\end{lem}

\begin{proof}
(1) Suppose that there are $R$-homomorphisms $\alpha,\beta:R^{n} \to R^{n}$ such that $\phi\circ\alpha+\beta\circ\psi=r I_n$.
Applying $\Hom_R(-,M)$ to the diagram 
\[
  \xymatrix@C=30pt@R=20pt{
    \cdots \ar[r]
    & R^{n} \ar[r]^\phi
    & R^{n} \ar[r]^\psi \ar[d]_{r} \ar[ld]_\alpha
    & R^{n} \ar[r] \ar[ld]_\beta
    & \cdots \\
    \cdots \ar[r]
    & R^{n} \ar[r]^\phi
    & R^{n} \ar[r]^\psi   
    & R^{n} \ar[r] 
    & \cdots, \\
  }
\]
we get a diagram
\[
  \xymatrix@C=30pt@R=20pt{
    \cdots \ar[r]
    & M^{n} \ar[r]^{^t\psi}
    & M^{n} \ar[r]^{^t\phi} \ar[d]_{r} \ar[ld]_{^t\beta}
    & M^{n} \ar[r] \ar[ld]_{^t\alpha}
    & \cdots \\
    \cdots \ar[r]
    & M^{n} \ar[r]^{^t\psi}
    & M^{n} \ar[r]^{^t\phi} 
    & M^{n} \ar[r] 
    & \cdots. \\
  }
\]
(The symbol $^t(-)$ denotes transpose.)
Then $\ker(^t\phi)/\image(^t\psi)\cong\Ext_R^2(M,M)\cong\Ext_R^1(\Omega M,\Omega^2 M)$ since $M\cong\Omega^2 M$.
The equality ${^t\alpha}\circ{^t\phi}+{^t\psi}\circ{^t\beta}=r I_n$ implies that $r$ annihilates the homology $\ker(^t\phi)/\image(^t\psi)$.
It follows from Lemmas \ref{DeyT3.8} and \ref{lemma AACE example1} that $r$ belongs to $\Ann(M)=\Ann(\Omega M)=\ann_R \Ext_R^1(\Omega M,\Omega^2 M)$.

Suppose $r\in\Ann(M)$.
Let $\pi:R^n\to M$ be the natural surjective, $\iota:M\to R^n$ the natural injection, and $\tilde{r}$ the multiplication by $r$ on $M$.
Since $\tilde{r}$ factors through a projective module and $\pi$ is surjective, there is an $R$-homomorphism $\gamma:M\to R^n$ such that $\tilde{r}=\pi\circ\gamma$.
We get an $R$-homomorphism $\alpha:R^n\to R^n$ such that $r I_n-\gamma\circ\pi=\phi\circ\alpha$ as $\image(r I_n-\gamma\circ\pi)$ is contained in $\ker\pi=\image\phi$.
Now $\cok\iota$ is maximal Cohen--Macaulay, which means that $\Hom_R(\iota, R^n)$ is surjective.
There is an $R$-homomorphism $\beta:R^n\to R^n$ such that $\gamma=\beta\circ\iota$.
Then we have $r I_n=\phi\circ\alpha+\gamma\circ\pi=\phi\circ\alpha+\beta\circ\psi$ by the equality $\psi=\iota\circ\pi$.

(2) Put $\alpha=(a_{ij})$ and $\beta=(b_{ij})$.
Then $r=\sum_{i=1}^n x_{ki} a_{ik}+\sum_{i=1}^n b_{ki} y_{ik}$ for all $1\le k\le n$.
\end{proof} 

In the rest of this section, let $k$ be an uncountable algebraically closed field having a characteristic different from 2.
For any complete local hypersurface $R$ with coefficient field $k$, it has countable CM-representation type if and only if it is isomorphic to $k \llbracket x_0, x_1, \ldots, x_d\rrbracket /(f)$, where $f$ is one of the following:
$$
(A_\infty)\ x_0^2+x_2^2+\cdots+x_d^2,\ {\rm or}\  (D_\infty)\ x_0^2 x_1+x_2^2+\cdots+x_d^2;
$$
see \cite[Theorem B]{BGS} or \cite[Theorem 14.16]{LW}.
The purpose of this section is to use this fact to prove Theorem \ref{main example countable}.
According to Kn\"{o}rrer's periodicity \cite[Theorem 3.1]{Kn} (see also \cite[Theorem 8.33]{LW}), the cases of one and two dimensions are essential.
Therefore, in the following, we calculate the cases for one and two dimensions.
(However, in the proof of the Theorem \ref{main example countable}, it is unnecessary for the case of dimension two.)
First, we deal with the case where $f=(A_\infty)$.

\begin{prop}\label{A type is compact}
\begin{enumerate}[\rm(1)]
\item Let $R=k \llbracket x, y\rrbracket/(x^2)$. Then $\Ann(\lcm(R))=\Ann(R/xR)=xR$ hold. 
\item Let $R=k \llbracket x, y, z\rrbracket/(x^2+z^2)$, and let $i$ be a element of $k$ such that $i^2+1=0$.
Then the equalities $\Ann(\lcm(R))=\Ann(R/(z-ix)R)=\Ann(R/(z+ix)R)=(x,z)R$ hold. 
\end{enumerate}
\noindent In particular, in both cases {\rm (1)} and {\rm (2)}, $\A(\lcm(R))$ is compact.
\end{prop}

\begin{proof}
(1) All the nonisomorphic nonfree indecomposable maximal Cohen--Macaulay $R$-modules are $R/xR$ and $\cok \phi_n$ for positive integers $n$, where 
$\phi_n=
\begin{pmatrix}
   x & y^n \\
   0 & -x
\end{pmatrix}$
by \cite[Proposition 4.1]{BGS}.
The ring $R$ is Gorenstein, and the complex $(\cdots \xrightarrow{x} R \xrightarrow{x} R \xrightarrow{x} \cdots)$ is exact.
By Lemma \ref{lemma AACE example2}(1), we get $\Ann(R/xR)=xR$.
Fix $n\ge 1$.
There is an exact sequence $(\cdots \xrightarrow{\phi_n} R^2 \xrightarrow{\phi_n} R^2 \xrightarrow{\phi_n} \cdots)$ of free $R$-modules.
Lemma \ref{lemma AACE example2}(2) deduces $\Ann(\cok \phi_n)\subseteq (x, y^n)R$.
On the other hand, there are equalities
$$
\phi_n
\begin{pmatrix}
   0 & 0 \\
   0 & -1
\end{pmatrix}
+
\begin{pmatrix}
   1 & 0 \\
   0 & 0
\end{pmatrix}
\phi_n
=xI_2
\ \ {\rm and}\ \
\phi_n
\begin{pmatrix}
   0 & 0 \\
   1 & 0
\end{pmatrix}
+
\begin{pmatrix}
   0 & 0 \\
   1 & 0
\end{pmatrix}
\phi_n
=y^n I_2,
$$
which means $\Ann(\cok \phi_n)=(x, y^n)R$ by Lemma \ref{lemma AACE example2}(1).
The assertion follows from Lemma \ref{lemma AACE example1}(1).

(2) For each positive integer $n$, we set 
$\psi_n^{+}=
\begin{pmatrix}
   z-ix & y^n \\
   0 & z+ix
\end{pmatrix}$
and 
$\psi_n^{-}=
\begin{pmatrix}
   z+ix & -y^n \\
   0 & z-ix
\end{pmatrix}$.
Then $R/(z-ix)R$, $R/(z+ix)R$, $\cok \psi_n^{+}$ and $\cok \psi_n^{-}$ for $n\ge 1$ are all the nonisomorphic nonfree indecomposable maximal Cohen--Macaulay $R$-modules; see \cite[Proposition 14.17]{LW} (or \cite[Theorem 5.3]{BD}).
The ring $R$ is Gorenstein, and the complex $(\cdots \xrightarrow{z+ix} R \xrightarrow{z-ix} R \xrightarrow{z+ix} \cdots)$ is exact.
By Lemma \ref{lemma AACE example2}(1), we have $\Ann(R/(z-ix)R)=\Ann(R/(z+ix)R)=(z-ix,z+ix)R=(x,z)R$.
Fix $n\ge 1$.
There is an exact sequence $(\cdots \xrightarrow{\psi_n^{-}} R^2 \xrightarrow{\psi_n^{+}} R^2 \xrightarrow{\psi_n^{-}} \cdots)$ of free $R$-modules.
It follows from Lemma \ref{lemma AACE example2}(2) that $\Ann(\cok \psi_n^{+})$ and $\Ann(\cok \psi_n^{-})$ are contained in $(z-ix,z+ix,y^n)R=(x,y^n, z)R$.
On the other hand, there are equalities
\begin{align*}
\psi_n^{+}
\begin{pmatrix}
   0 & 0 \\
   0 & 1
\end{pmatrix}
+
\begin{pmatrix}
   1 & 0 \\
   0 & 0
\end{pmatrix}
\psi_n^{-}
=(z+ix),\ \
& \psi_n^{+}
\begin{pmatrix}
   1 & 0 \\
   0 & 0
\end{pmatrix}
+
\begin{pmatrix}
   0 & 0 \\
   0 & 1
\end{pmatrix}
\psi_n^{-}
=(z-ix) I_2, \\
 {\rm and} \ \
& \psi_n^{+}
\begin{pmatrix}
   0 & 0 \\
   1 & 0
\end{pmatrix}
+
\begin{pmatrix}
   0 & 0 \\
   -1 & 0
\end{pmatrix}
\psi_n^{-}
=y^n I_2,
\end{align*}
which means $\Ann(\cok \psi_n^{+})=\Ann(\cok \psi_n^{-})=(x,y^n, z)R$ by Lemmas \ref{lemma AACE example1}(2) and \ref{lemma AACE example2}(1).
Similar to the proof of (1), Lemma \ref{lemma AACE example1}(1) concludes that the assertion holds.
\end{proof} 

\begin{rem}\label{A type non compact cm0}
Let us verify that the nine conditions of Theorem \ref{main TFAE} do not hold for $R$ in Proposition \ref{A type is compact}(1).
Put $\m=(x,y)R$.
Note that $R$ is complete and $\sing(R)=\V(xR)$, which means $R$ does not have an isolated singularity.
Proposition \ref{A type is compact}(1) asserts that $\Ann(\lcm(R))$ is not $\m$-primary.
All the nonisomorphic nonfree indecomposable maximal Cohen--Macaulay $R$-modules which are locally free on the punctured spectrum are $\cok \phi_n$ for $n\ge 1$ since $\Ann(R/xR)=xR$ and $\Ann(\cok \phi_n)=(x, y^n)R$.
Then 
$$
\Ann(\lcm_0(R))=\bigcap_{n\ge 1} \Ann(\cok \phi_n)=\bigcap_{n\ge 1} (x, y^n)R=xR.
$$
So, $\Ann(\lcm_0(R))$ is also not $\m$-primary, and there is no $0\ne M\in\lcm_0(R)$ such that $\Ann(\lcm_0(R))=\Ann(M)$ as $M$ is isomorphic to a finite direct sum of $\cok \phi_n$. 
For any positive integer $n$, we have $\Omega^{n-1}(R/xR)\cong R/xR$ and hence
$\Ext_R^n(R/xR, R/xR)\cong\Ext_R^1(R/xR, R/xR)$, which implies
$$
\ca^n(R) \subseteq \ann_R \Ext_R^n(R/xR, R/xR)= \ann_R \Ext_R^1(R/xR, R/xR)=\Ann(R/xR)=xR.
$$
We see that $\ca(R)$ is contained in $xR$ and is not $\m$-primary.
(In general, for a Gorenstein local ring $R$, $\ca(R)=\Ann(\lcm(R))$; see \cite[Lemma 2.3]{E} for instance.)
A similar argument for the rings in Proposition \ref{A type is compact}(2) and Proposition \ref{D type is compact} shows that the conditions of Theorem \ref{main TFAE} are not satisfied.
\end{rem}

Next, we consider the case where $f=(D_\infty)$.

\begin{prop}\label{D type is compact}
\begin{enumerate}[\rm(1)]
\item Let $R=k \llbracket x, y\rrbracket/(x^2 y)$. Then $\Ann(\lcm(R))=
\Ann(R/xR\oplus R/y R)=(x^2,xy)R$. 
\item Let $R=k \llbracket x, y, z\rrbracket/(x^2 y+z^2)$.
Consider the following two $R$-homomorphisms
$$
\alpha=
\begin{pmatrix}
   z & -y \\
   x^2 & z
\end{pmatrix}
\ \ {\rm and} \ \
\beta=
\begin{pmatrix}
   z & -xy \\
   x & z
\end{pmatrix}.
$$
Then the equalities $\Ann(\lcm(R))=\Ann(\cok\alpha\oplus\cok\beta)=(x^2, xy,z)R$ hold. 
\end{enumerate}
\noindent In particular, in both cases {\rm (1)} and {\rm (2)}, $\A(\lcm(R))$ is compact.
\end{prop}

\begin{proof}
(1) All the nonisomorphic nonfree indecomposable maximal Cohen--Macaulay $R$-modules are $R/xR$, $R/xyR$, $R/yR$, $R/x^2R$, $\cok\alpha_n$, $\cok\beta_n$, $\cok\gamma_n$, and $\cok\delta_n$ for positive integers $n$, where 
$\alpha_n=
\begin{pmatrix}
   xy & y^n \\
   0 & -x
\end{pmatrix}$,
$\beta_n=
\begin{pmatrix}
   x & y^n \\
   0 & -xy
\end{pmatrix}$,
$\gamma_n=
\begin{pmatrix}
   x & y^n \\
   0 & -x
\end{pmatrix}$,
$\delta_n=
\begin{pmatrix}
   xy & y^{n+1} \\
   0 & -xy
\end{pmatrix}$
by \cite[Proposition 4.2]{BGS}.
The ring $R$ is Gorenstein, and $(\cdots \xrightarrow{xy} R \xrightarrow{x} R \xrightarrow{xy} \cdots)$ and $(\cdots \xrightarrow{x^2} R \xrightarrow{y} R \xrightarrow{x^2} \cdots)$ are exact.
By Lemma \ref{lemma AACE example2}(1), we get $\Ann(R/xR)=\Ann(R/xyR)=xR$ and $\Ann(R/yR)=\Ann(R/x^2R)=(x^2,y)R$.
Therefore, we have $\Ann(R/xR\oplus R/y R)=\Ann(R/xR)\cap\Ann(R/yR)=xR\cap(x^2,y)R=(x^2,xy)R$.
Fix $n\ge 1$.
There are exact sequences $(\cdots \xrightarrow{\beta_n} R^2 \xrightarrow{\alpha_n} R^2 \xrightarrow{\beta_n} \cdots)$ and $(\cdots \xrightarrow{\delta_n} R^2 \xrightarrow{\gamma_n} R^2 \xrightarrow{\delta_n} \cdots)$ of free $R$-modules.
It follows from Lemma \ref{lemma AACE example2}(2) that $\Ann(\cok \alpha_n)\subseteq (x, y^n)R$ and $\Ann(\cok \gamma_n)\subseteq (x, y^{n+1})R$.
On the other hand, there are equalities
\begin{align*}
\alpha_n
\begin{pmatrix}
   0 & 0 \\
   0 & -1
\end{pmatrix}
+
\begin{pmatrix}
   1 & 0 \\
   0 & 0
\end{pmatrix}
\beta_n
=xI_2, \ \
& \alpha_n
\begin{pmatrix}
   0 & 0 \\
   1 & 0
\end{pmatrix}
+
\begin{pmatrix}
   0 & 0 \\
   1 & 0
\end{pmatrix}
\beta_n
=y^n I_2, \\
\gamma_n
\begin{pmatrix}
   0 & 0 \\
   0 & -y
\end{pmatrix}
+
\begin{pmatrix}
   1 & 0 \\
   0 & 0
\end{pmatrix}
\delta_n
=xyI_2, \ \
& \gamma_n
\begin{pmatrix}
   0 & 0 \\
   y & 0
\end{pmatrix}
+
\begin{pmatrix}
   0 & 0 \\
   1 & 0
\end{pmatrix}
\delta_n
=y^{n+1}I_2, \\
 {\rm and} \ \
& \gamma_n
\begin{pmatrix}
   x & 0 \\
   0 & -x
\end{pmatrix}
+
\begin{pmatrix}
   0 & -y^{n-1} \\
   0 & 0
\end{pmatrix}
\delta_n
=x^2 I_2,
\end{align*}
which means that $\Ann(\cok \alpha_n)=\Ann(\cok \beta_n)=(x, y^n)R$ and $(x^2,xy, y^{n+1})R\subseteq\Ann(\cok \gamma_n)=\Ann(\cok \delta_n)\subseteq (x, y^{n+1})R$ by Lemmas \ref{lemma AACE example1}(2) and \ref{lemma AACE example2}(1).
The assertion follows from Lemma \ref{lemma AACE example1}(1).

(2) All the nonisomorphic nonfree indecomposable maximal Cohen--Macaulay $R$-modules are $\cok\alpha^{+}$, $\cok\alpha^{-}$, $\cok\beta^{+}$, $\cok\beta^{-}$, $\cok\gamma_n^{+}$, $\cok\gamma_n^{-}$, $\cok\delta_n^{+}$, and $\cok\delta_n^{-}$ for positive integers $n$, where 
\begin{align*}
\alpha^{+}=
\begin{pmatrix}
   z & y \\
   -x^2 & z
\end{pmatrix}, \
\alpha^{-}=
\begin{pmatrix}
   z & -y \\
   x^2 & z
\end{pmatrix}, \
&\beta^{+}=
\begin{pmatrix}
   z & xy \\
   -x & z
\end{pmatrix}, \
\beta^{-}=
\begin{pmatrix}
   z & -xy \\
   x & z
\end{pmatrix}, \\
\gamma_n^{+}=
\begin{pmatrix}
   z & 0 & xy & 0  \\
   0 & z & y^{n+1} & -x  \\
   -x & 0 & z & 0  \\
   -y^{n+1} & xy & 0 & z
\end{pmatrix}, \
&\gamma_n^{-}=
\begin{pmatrix}
   z & 0 & -xy & 0  \\
   0 & z & -y^{n+1} & x  \\
   x & 0 & z & 0  \\
   y^{n+1} & -xy & 0 & z
\end{pmatrix}, \\
\delta_n^{+}=
\begin{pmatrix}
   z & 0 & xy & 0  \\
   0 & z & y^{n+1} & -xy  \\
   -x & 0 & z & 0  \\
   -y^n & x & 0 & z
\end{pmatrix}, \
& {\rm and} \ \delta_n^{-}=
\begin{pmatrix}
   z & 0 & -xy & 0  \\
   0 & z & -y^{n+1} & xy  \\
   x & 0 & z & 0  \\
   y^n & -x & 0 & z
\end{pmatrix};
\end{align*}
see \cite[Proposition 14.19]{LW} (or \cite[Theorem 5.7]{BD}).
For any integer $n\ge 1$, there are exact sequences 
\begin{align*}
(\cdots \xrightarrow{\alpha^{+}} R^2 \xrightarrow{\alpha^{-}} R^2 \xrightarrow{\alpha^{+}} R^2 \xrightarrow{\alpha^{-}} \cdots),\ 
&(\cdots \xrightarrow{\beta^{+}} R^2 \xrightarrow{\beta^{-}} R^2 \xrightarrow{\beta^{+}} R^2 \xrightarrow{\beta^{-}} \cdots), \\
(\cdots \xrightarrow{\gamma_n^{+}} R^4 \xrightarrow{\gamma_n^{-}} R^4 \xrightarrow{\gamma_n^{+}} R^4 \xrightarrow{\gamma_n^{-}} \cdots), \ {\rm and} \
&(\cdots \xrightarrow{\delta_n^{+}} R^4 \xrightarrow{\delta_n^{-}} R^4 \xrightarrow{\delta_n^{+}} R^4 \xrightarrow{\delta_n^{-}} \cdots)
\end{align*}
of free $R$-modules.
Since $R$ is Gorenstein, it follows from Lemmas \ref{lemma AACE example1}(2) and \ref{lemma AACE example2}(2) that
\begin{align*}
\Ann(\cok \alpha^{+})=\Ann(\cok \alpha^{-})\subseteq(x^2, y, z)R, \ 
&\Ann(\cok \beta^{+})=\Ann(\cok \beta^{-})\subseteq(x, z)R, \\
\Ann(\cok \gamma_n^{+})=\Ann(\cok \gamma_n^{-})\subseteq(x, y^{n+1}, z)R, \ {\rm and} \ 
&\Ann(\cok \delta_n^{+})=\Ann(\cok \delta_n^{-})\subseteq(x, y^{n+1}, z)R.
\end{align*}
On the other hand, there are equalities
\begin{align*}
& \alpha^{+}
\begin{pmatrix}
   0 & -1 \\
   0 & 0
\end{pmatrix}
+
\begin{pmatrix}
   0 & 1 \\
   0 & 0
\end{pmatrix}
\alpha^{-}
=x^2 I_2, \ \
\alpha^{+}
\begin{pmatrix}
   0 & 0 \\
   1 & 0
\end{pmatrix}
+
\begin{pmatrix}
   0 & 0 \\
   -1 & 0
\end{pmatrix}
\alpha^{-}
=y I_2, \\ 
& \alpha^{+}
\begin{pmatrix}
   1 & 0 \\
   0 & 0
\end{pmatrix}
+
\begin{pmatrix}
   0 & 0 \\
   0 & 1
\end{pmatrix}
\alpha^{-}
=zI_2, \ \
\beta^{+}
\begin{pmatrix}
   0 & -1 \\
   0 & 0
\end{pmatrix}
+
\begin{pmatrix}
   0 & 1 \\
   0 & 0
\end{pmatrix}
\beta^{-}
=x I_2, \\ 
& \beta^{+}
\begin{pmatrix}
   1 & 0 \\
   0 & 0
\end{pmatrix}
+
\begin{pmatrix}
   0 & 0 \\
   0 & 1
\end{pmatrix}
\beta^{-}
=zI_2, \ \
\gamma_n^{+}
\begin{pmatrix}
   0 & 0 & -1 & 0  \\
   0 & 0 & 0 & 0  \\
   0 & 0 & 0 & 0  \\
   0 & -1 & 0 & 0
\end{pmatrix}
+
\begin{pmatrix}
   0 & 0 & 1 & 0  \\
   0 & 0 & 0 & 0  \\
   0 & 0 & 0 & 0  \\
   0 & 1 & 0 & 0
\end{pmatrix}
\gamma_n^{-}
=xI_4, \\ 
& \gamma_n^{+}
\begin{pmatrix}
   0 & 0 & 0 & -1  \\
   0 & 0 & 0 & 0  \\
   0 & 1 & 0 & 0  \\
   0 & 0 & 0 & 0
\end{pmatrix}
+
\begin{pmatrix}
   0 & 0 & 0 & 1  \\
   0 & 0 & 0 & 0  \\
   0 & -1 & 0 & 0  \\
   0 & 0 & 0 & 0
\end{pmatrix}
\gamma_n^{-}
=y^{n+1}I_4, \
\gamma_n^{+}
\begin{pmatrix}
   1 & 0 & 0 & 0  \\
   0 & 1 & 0 & 0  \\
   0 & 0 & 0 & 0  \\
   0 & 0 & 0 & 0
\end{pmatrix}
+
\begin{pmatrix}
   0 & 0 & 0 & 0  \\
   0 & 0 & 0 & 0  \\
   0 & 0 & 1 & 0  \\
   0 & 0 & 0 & 1
\end{pmatrix}
\gamma_n^{-}
=zI_4, \\ 
& \delta_n^{+}
\begin{pmatrix}
   0 & 0 & \hspace{-0.6em} -x \hspace{-0.6em} & 0  \\
   0 & 0 & 0 & x  \\
   0 & 0 & 0 & 0  \\
   -y^{n-1} \hspace{-0.8em} & 0 & 0 & 0
\end{pmatrix}
+
\begin{pmatrix}
   0 & 0 & \hspace{-0.2em} x & 0  \\
   0 & 0 & \hspace{-0.2em} 0 & \hspace{-0.7em} -x \hspace{-0.3em}\\
   0 & 0 & \hspace{-0.2em} 0 & 0  \\
   y^{n-1} \hspace{-0.8em} & 0 & \hspace{-0.2em} 0 & 0
\end{pmatrix}
\delta_n^{-}
=x^2 I_4, \ 
\delta_n^{+}
\begin{pmatrix}
   0 & 0 & 0 & 0  \\
   0 & 0 & 0 & y  \\
   1 & 0 & 0 & 0  \\
   0 & 0 & 0 & 0
\end{pmatrix}
+
\begin{pmatrix}
   0 & 0 & 0 & 0  \\
   0 & 0 & 0 & \hspace{-0.7em} -y \hspace{-0.3em} \\
   \hspace{-0.3em} -1\hspace{-0.5em} & 0 & 0 & 0  \\
   0 & 0 & 0 & 0
\end{pmatrix}
\delta_n^{-}
=xyI_4, \\ 
& \delta_n^{+}
\begin{pmatrix}
   0 & \hspace{-0.2em} 0 & \hspace{-0.2em} 0 & \hspace{-0.5em} -y \\
   0 & \hspace{-0.2em} 0 & \hspace{-0.2em} 0 & 0  \\
   0 & \hspace{-0.2em} 1 & \hspace{-0.2em} 0 & 0  \\
   0 & \hspace{-0.2em} 0 & \hspace{-0.2em} 0 & 0
\end{pmatrix}
+
\begin{pmatrix}
   0 & 0 & 0 \hspace{-0.2em} & y  \\
   0 & 0 & 0 \hspace{-0.2em} & 0  \\
   0 & \hspace{-0.5em} -1 \hspace{-0.3em} & 0 \hspace{-0.2em} & 0  \\
   0 & 0 & 0 \hspace{-0.2em} & 0
\end{pmatrix}
\delta_n^{-}
=y^{n+1}I_4, \ {\rm and} \
\delta_n^{+}
\begin{pmatrix}
   1 & \hspace{-0.2em} 0 & \hspace{-0.2em} 0 & \hspace{-0.2em} 0  \\
   0 & \hspace{-0.2em} 1 & \hspace{-0.2em} 0 & \hspace{-0.2em} 0  \\
   0 & \hspace{-0.2em} 0 & \hspace{-0.2em} 0 & \hspace{-0.2em} 0  \\
   0 & \hspace{-0.2em} 0 & \hspace{-0.2em} 0 & \hspace{-0.2em} 0
\end{pmatrix}
+
\begin{pmatrix}
   0 & \hspace{-0.2em} 0 & \hspace{-0.2em} 0 & \hspace{-0.2em} 0  \\
   0 & \hspace{-0.2em} 0 & \hspace{-0.2em} 0 & \hspace{-0.2em} 0  \\
   0 & \hspace{-0.2em} 0 & \hspace{-0.2em} 1 & \hspace{-0.2em} 0  \\
   0 & \hspace{-0.2em} 0 & \hspace{-0.2em} 0 & \hspace{-0.2em} 1
\end{pmatrix}
\delta_n^{-}
=zI_4. 
\end{align*}
We obtain $\Ann(\cok \alpha^{+})=(x^2, y, z)R$, $\Ann(\cok \beta^{+})=(x, z)R$, $\Ann(\cok \gamma_n^{+})=(x, y^{n+1}, z)R$, and $(x^2, xy, y^{n+1}, z)R\subseteq\Ann(\cok \delta_n^{+})\subseteq(x, y^{n+1}, z)R$ by Lemma \ref{lemma AACE example2}(1).
Similar to the proof of (1), Lemma \ref{lemma AACE example1}(1) concludes that the assertion holds.
\end{proof} 

\begin{rem}\label{Strict equality sign}
With the notation of the proof of Proposition \ref{D type is compact}(1) one has the equalities $\Ann(\cok \gamma_n)=\Ann(\cok \delta_n)=(x^2,xy, y^{n+1})R$.
We put $J=\Ann(\cok \gamma_n)=\Ann(\cok \delta_n)$.
Since the quotient module $(x, y^{n+1})R/(x^2,xy, y^{n+1})R$ is isomorphic to $k$, it suffices to show $J\subsetneq (x, y^{n+1})R$, that is, $xR\nsubseteq J$.
If $xR\subseteq J$, there are $a,b,c,d,e,f,g,h\in k \llbracket x, y\rrbracket$ such that 
$$
\gamma_n
\begin{pmatrix}
   a & b \\
   c & d
\end{pmatrix}
+
\begin{pmatrix}
   e & f \\
   g & h
\end{pmatrix}
\delta_n
=xI_2
$$
as an $R$-endomorphism of $R^2$ by Lemma \ref{lemma AACE example2}(1).
Comparing the $(2,2)$ entry of the matrix, we have $(-dx+gy^{n+1}-hxy)-x\in (x^2 y)$.
This means $(d+1)x\in (y)$ and thus $d$ is unit.
Comparing the $(1,2)$ entry, one has $bx+dy^n+ey^{n+1}-fxy \in (x^2 y)$, which deduces $dy^n\in (x, y^{n+1})$, a contradiction.
Similarly, with the notation of the proof of Proposition \ref{D type is compact}(2) one has the equalities $\Ann(\cok \delta_n^{+})=\Ann(\cok \delta_n^{-})=(x^2, xy, y^{n+1}, z)R$.
\end{rem}

The last main theorem of this paper is proved by the same method as in \cite[Section 5]{AACE}.
Note that the assumption in \cite[Subsection 5.1]{AACE} that the characteristic of the residue field is 0 can be replaced with the assumption that it is not 2; see \cite[Propositions 8.15 and 8.18]{LW} for instance.

\begin{proof}[Proof of Theorem \ref{main example countable}]
Put $S=k \llbracket x_0, x_1, \ldots, x_d\rrbracket$, $R=S/(f)$ and $R^\#=S\llbracket z \rrbracket/(f+z^2)$ where $0\ne f\in (x_0, x_1, \ldots, x_d)^2$.
It follows from Propositions \ref{A type is compact} and \ref{D type is compact} that we have only to show that $\A(\lcm(R^\#))$ is compact if $\A(\lcm(R))$ is compact.
Suppose that $\A(\lcm(R))$ is compact.
By \cite[Theorem 2.6]{AACE}, there exists a maximal Cohen--Macaulay $R$-module $M$ such that $\Ann(\lcm(R))=\Ann(M)$, which means that $M$ belongs to any closed subset of $\A(\lcm(R))$.
We see by \cite[(the proof of) Proposition 5.4]{AACE} that $\Omega_{R^\#} M$ belongs to any closed subset of $\A(\lcm(R^\#))$.
Therefore, we have $\Ann(\lcm(R^\#))=\Ann(\Omega_{R^\#} M)$.
Again, by \cite[Theorem 2.6]{AACE}, we conclude that $\A(\lcm(R^\#))$ is compact.
\end{proof}

\begin{ac}
The author would like to thank his supervisor Ryo Takahashi for valuable comments.
\end{ac}

\end{document}